\documentclass[11pt, reqno]{amsart}
\usepackage{euscript,amscd,amsgen,amsfonts,amssymb,latexsym}

\newcommand{\eqdef}{\stackrel{\scriptscriptstyle\rm def}{=}}
\usepackage{bbm}
\usepackage{pgfplots, multicol}
\usepackage{mathtools}
\usepackage{paralist}
\usepackage{todonotes}
\usepackage{enumerate}
\usepackage{array}
\usepackage{graphicx}

                 \def\RR{\mathbb{R}}        \def\ZZ{\mathbb{Z}}

\newtheorem{theorem}{Theorem}[section]

\newtheorem{proposition}[theorem]{Proposition}
\newtheorem{corollary}[theorem]{Corollary}

\newtheorem{lemma}[theorem]{Lemma}

\newtheorem{property}[theorem]{Property}

\newcommand{\beha}{\begin{enumerate}}
\newcommand{\behe}{\end{enumerate}}
\renewcommand{\epsilon}{\varepsilon}

\newcommand{\Per}{{\rm Per}}

\newcommand{\Or}{\mathcal{O}}

\newcommand{\cM}{\EuScript{M}}

\newcommand{\cH}{\EuScript{H}}
\newcommand{\bR}{{\mathbb R}}
\newcommand{\bZ}{{\mathbb Z}}
\newcommand{\bN}{{\mathbb N}}

\newcommand{\cB}{{\mathcal B}}
\newcommand{\cC}{{ \mathcal{C}}}
\newcommand{\cK}{{\mathcal K}}

\newcommand{\cU}{{\mathcal U}}

\newcommand{\cV}{{\mathcal V}}
\newcommand{\cL}{{\mathcal L}}
\newcommand{\cR}{{\mathcal R}}

\newcommand{\cS}{{\mathcal S}}

\let\d=\delta

\def\1{1\!\!1}

\def\and{\text{ and }}

        \def\conv{\text{{\rm conv}}}

                             \def\d{\delta}
                         
                           \def\l{\lambda}

                        \def\^{\widetilde}

\def\Per{{\rm Per}}
\def\inn{{\rm int}}

\def\Per{{\rm Per}}

\def\1{1\!\!1}

\def\rv{{\rm rv}}
\def\inn{{\rm int\ }}

\newtheorem*{thmMain}{Main Theorem}

\usepackage{xcolor}

\DeclareMathSymbol{\varnothing}{\mathord}{AMSb}{"3F}
\renewcommand{\emptyset}{\varnothing}

\title{A horseshoe with a discontinuous entropy spectrum}

\thanks{P.J. and J.W. were supported by a Dr. Barnett and Jean Hollander Rich CCNY Mathematics Scholarship.} 
\thanks{C.W. was partially supported by  grants from the PSC-CUNY (TRADB-49-253 to Christian Wolf) and the Simons Foundation (\#637594 to Christian Wolf).}

\author{Pavel Javornik}\address{Department of Mathematics, The City College of New York, New York, NY, 10031, USA}\email{pjavorn000@citymail.cuny.edu}
\author{Joseph Winter}\address{Department of Mathematics, The City College of New York, New York, NY, 10031, USA}\email{jwinter@citymail.cuny.edu}

\author{Christian Wolf}\address{Department of Mathematics, The City College of New York, New York, NY, 10031, USA}\email{cwolf@ccny.cuny.edu}

\begin{document}

\begin{abstract}
We study the regularity of the entropy spectrum of the Lyapunov exponents for hyperbolic maps on surfaces. It is well-known that the entropy spectrum is a concave upper semi-continuous function which is analytic on the interior of the set Lyapunov exponents. In this paper we construct a
family of horseshoes with a discontinuous entropy spectrum at the boundary of the set of Lyapunov exponents.
\end{abstract}
\maketitle

\section{Introduction}

\subsection{Motivation}
 It is one of the central objectives in the  multifractal analysis of dynamical systems to study relevant  characteristics of complexity (entropy, pressure, dimension, etc.) on  level sets of various dynamically defined functions, see e.g. \cite{BPS} for an introduction of the concepts. In the context of differentiable dynamics, arguably the most important level sets are given by the pointwise Lyapunov exponents since they provide a quantitative measurement for the chaotic nature of the system. In this paper we consider hyperbolic diffeomorphisms $f$ on surfaces (henceforth called horseshoe maps) and study the regularity of the entropy spectrum $\cH_f$ on set of Lyapunov exponents $\cR_L(f)$. These spectra are of interest in higher-dimensional multifractal analysis \cite{BSS,C, FFW}.

  It is frequently observed for hyperbolic systems that the spectrum is analytic in the interior and continuous at the boundary of the domain. In particular,   for one-dimensional spectra continuity   follows from the simple fact that  concave and upper semi-continuous functions  are continuous. For higher dimensional spectra the situation is more complicated, in part, since concave and upper semi-continuous functions may have in theory discontinuities at the boundary of the domain unless the domain is   polyhedral \cite{GKR}.
We note that the case $\cR_L(f)$ being a  polyhedron occurs for example  if $f$ is piecewise linear \cite{Z}, in which case  $\cH_f$ is continuous on $\cR_L(f)$.
In this paper we show that  the entropy spectrum may be discontinuous even when the polyhedron property fails only at one point. Namely, we construct a $C^2$-horseshoe map $f$ with a discontinuous entropy spectrum for which $\cR_L(f)$ is a polygon with infinitely many vertices where all but one of the vertices are isolated.  We construct this example  as a limit of  piecewise linear horseshoe maps $f_k$ that go through a series of local perturbations. Our example is motivated by the recently obtained discontinuity of the entropy spectrum for shift maps and certain Lipschitz  potentials \cite{Wo}. Other than in the symbolic case where one can directly define the appropriate potential  on the shift space, in the smooth case  the potential is indirectly defined through the expansion/contraction rates on the horseshoe  which leads to various technical complications in our construction. 
 \subsection{Statement of the result}Let $f:M\to M$ be a $C^{1+\epsilon}$ diffeomorphism on a smooth surface $M$, and let $\Lambda\subset M$ be a basic set of $f$, that is, $\Lambda$ is a compact locally maximal hyperbolic set of $f$ such $f\vert_\Lambda$ is topologically mixing. Let  $T_\Lambda f =E^u\oplus E^s$ denote the hyperbolic splitting of the tangent space of $\Lambda$. 
Let $\cM$ denote the set of  $f$-invariant Borel probability measures on $\Lambda$, and let $\cM_E\subset \cM$ be the subset of ergodic measures. Recall that $\cM$ endowed with the weak$^\ast$ topology is a compact convex metrizable topological space. We define $\Phi_L=\Phi_{L}(f)=(-\phi^s,\phi^u): \Lambda \to\bR^2$ where $\phi^{s/u}=\log \|Df|_{E^{s/u}}\|$.  
Recall that $-\phi^{s}$ and $\phi^u$ are positive\footnote{After a possible change to an adapted Riemann metric.}  Lipschitz continuous functions.
 We define 
\begin{equation}
\cR_L(f)=\cR_L(f,\Phi_L)=\{\rv(\mu):\mu\in \cM\},
\end{equation}
 where $\rv(\mu)=(-\int \phi^s\, d\mu, \int \phi^u\, d\mu)$. It follows that $\cR_L(f)$ is a compact and convex set that is contained in the interior of the first quadrant of $\bR^2$.  The entropy spectrum of $\cR_L(f)$   is defined 
by  
\begin{equation}\label{defcH}
\cH_f(w)=\sup \{h_\mu(f): \rv(\mu)=w\},
\end{equation}
where $h_\mu(f)$ denotes the measure-theoretic entropy of $\mu$.  We note that $\cH_f|_{\inn \cR_L(f)}$  is  analytic and coincides with the entropy of the  level  sets of Lyapunov exponents \cite{BSS,C}. More precisely, given $w=(w_1,w_2)\in  \inn \cR_L(f)$ we have
$\cH_f(w)= h(f|_{\cK_{w_1,w_2}})$,  where 
\[
\cK_{w_1,w_2}=\left\{x\in \Lambda: \lim_{n\to \infty}\frac{1}{n} \log \Vert Df^{\pm n}(x)|_{E^{s/u}}\Vert =w_{1/2}\right\}\]
and  $h(f|_{\cK_{w_1,w_2}})$ is the entropy of $f$ restricted to the (in general) non-compact invariant set $\cK_{w_1,w_2}$ (see \cite{Bo}). Recall that $\cK_{w_1,w_2}$ is the set of points $x\in \Lambda$ with Lyapunov exponents $-w_1<0<w_2$. It is straight-forward to verify that $\cH_f$ is concave and upper semi-continuous. This suggests that $\cH_f$ might be continuous. However, somewhat unexpectedly, we are able to show the 
 following result.

\begin{thmMain}
There exists a  $C^2$-surface diffeomorphism $f$ with a basic set $\Lambda$, such that $f|_\Lambda$ is conjugate to a  full shift with $3$ symbols, with the following properties:
\begin{enumerate}
\item[\rm (i)] The set $\cR_L(f)$ has non-empty interior and countably many extreme points of which all but one are isolated;
 \item[\rm (ii)] The entropy spectrum $w\mapsto \cH_f(w)$ is discontinuous at the non-isolated extreme point.
 \end{enumerate}
\end{thmMain}

We  note that the reason for formulating our theorem for hyperbolic sets whose symbolic representation is given by a full-shift in  $3$ symbols   is for the  ease  of presentation. Our construction can be modified to obtain similar discontinuity results for other hyperbolic sets of surface diffeomorphisms. Note that the boundary of $\cR_L(f)$ in the main theorem is  a countable polygon with precisely one non-isolated vertex point. This shows that  $\cR_L(f)$ is of the simplest possible shape for which $w\mapsto \cH_f(w)$ can be discontinuous. Indeed, it is a consequence of the celebrated Gale, Klee and Rockafellar theorem \cite{GKR} that that every concave upper semi-continuous function with polyhedral domain is continuous. In particular, in our example $w\mapsto \cH_f(w)$ is continuous everywhere except at the non-isolated extreme point of  $\cR_L(f)$. The strategy to prove the Main Theorem is to apply a similar recently obtained discontinuity result  for one-sided shift maps \cite{Wo} and to construct a horseshoe map for which the entropy spectrum of the Lyapunov exponents coincides with that of  the symbolic system. To accomplish this, we consider a piecewise linear horseshoe and perform
countably many perturbations (which we call surgeries) to obtain a $C^2$-horseshoe that satisfies the assertions in our main theorem. 

Finally we mention the possible applicability of our methods to establish further properties of various spectra
for smooth systems. One example of such an application is  the study of the geometric shape of the set of Lyapunov exponents $\cR_L(f)$ (cf.  \cite{KW} for related result for symbolic systems).

This paper is organized as follows. In Section 2 we recall some basic notation from ergodic theory and symbolic dynamics. Moreover, we discuss a symbolic example of a discontinuous localized entropy function. Section 3 is devoted to the construction of a horseshoe map $f=\lim_{k\to\infty} f_k$ that satisfies the assertions of Theorem \ref{thm32}. Finally, in the section 4 we complete the proof of our Main Theorem.

 \section{Settings and background material.}
 In this section we introduce the relevant background material.  In particular, we provide an overview of the pertinent results and definitions from symbolic dynamics and discuss a recent result concerning a discontinuous entropy spectrum for shift maps.
\subsection{Shift maps}
We start by reviewing some material from symbolic dynamics. 
We will discuss simultaneously one-sided and two-sided shift maps.
We denote by $\bZ$ the set of all integers and by $\bN=\{0,1,2,3,\cdots\}$ the set of all non-negative integers. Let $d\geq 2$ and  $E=\{0,\dots,d-1\}$. We define
\begin{align*}
\Sigma^+_d&=\left\{(\xi_k)_0^{+\infty}: \xi_k\in E\right\},\,\, \text{and}\\
\Sigma^\pm_d&=\left\{(\xi_k)_{-\infty}^{+\infty}: \xi_k\in E\right\}
\end{align*}
We refer to either of these sets as a shift or a symbol space. In the case when we do not
want to specify the shift space or also when it is clear from the context which shift space
is meant, we write $\Sigma$ instead of $\Sigma^+$
 or $\Sigma^\pm$. We endow $\Sigma$ with the {\em Tychonov product }topology
which makes $\Sigma$ a compact metrizable space. For example, given $0<\theta<1$, the metric given by
\[d(\xi,\eta)=d_\theta(\xi,\eta)\eqdef\theta^{\min\{k\in \bN:\  \xi_k\not=\eta_k\, \text{or}\,    \xi_{-k}\not=\eta_{-k}   \}}\]
induces the Tychonov product topology on $\Sigma$. Here we use the common convention that $\theta^{+\infty}=0$.
The (left) {\em shift map} $\sigma:\Sigma\to \Sigma$ is defined by $\sigma(\xi)_k=\xi_{k+1}$. Occasionally, in order to avoid confusion, we write $\sigma^+$ for $\sigma : \Sigma^+\to 
\Sigma^+$ and $\sigma^{\pm}$ for  $\sigma : \Sigma^\pm\to 
\Sigma^\pm$.
Note that in the case of $\Sigma^\pm$, the shift map is injective and in the
case of $\Sigma^+$, the shift map is a $d$ to one maps that performs cutting off the zero$^{\rm th}$ coordinate.
For $i\leq j$ and $t=t_it_{i+1} \cdots t_{j}\in E^{j-i+1}$   
we denote the  cylinder generated by $t$ by \[\cC_{i,j}(t)=\{\xi\in \Sigma: \xi_i=t_i,\dots, \xi_{j}=t_{j}\}.\]  Moreover, for  $\xi\in \Sigma$ and $i\leq j$ we write 
$\xi\vert_i^j=\xi_i\cdots\xi_{j}$ and  call $\cC_{i,j}(\xi)=\cC_{i,j}(\xi_i\cdots \xi_j)$ the  cylinder starting at $i$ of length $j-i+1$ generated by $\xi$. If $i=0$ we frequently write $\xi\vert^j$ 
instead of $\xi \vert^j_0$ and $\cC_{j}(\xi)$ instead of $\cC_{i,j}(\xi)$. 
We denote by
$\Pi: \Sigma^\pm\to \Sigma^+$ 
the projection from $\Sigma^\pm$ to $\Sigma^{+}$ defined by
\begin{equation}\label{defproj}
\Pi\left((\xi_k)_{-\infty}^{+\infty}\right)=(\xi_k)_0^{+\infty}.
\end{equation}
Further, given $t=t_0t_1\cdots t_{n-1}$ we call  \[\Or(t)=\cdots t_0\cdots t_{n-1}t_0\cdots t_{n-1}t_0\cdots t_{n-1}\cdots \in \Sigma^\pm\] respectively \[\Or(t)= t_0\cdots t_{n-1}t_0\cdots t_{n-1}t_0\cdots t_{n-1}\cdots \in \Sigma^{+}\] the periodic point of $\sigma$  generated by $t$ (of period $n$).  For a periodic point $\xi$ with prime period $n$  we call $\tau_\xi=\xi_0\cdots \xi_{n-1}$ the generating segment of  $\xi$, that is, $\xi=\Or(\tau_\xi)$.  
\subsection{Notions from ergodic theory}
Let  $f:X\to X$ be a continuous map on a compact metric space $X$. Let $\cM$ denote the set of all $f$-invariant Borel probability measures on $X$ endowed with the weak$^\ast$ topology, and let $\cM_E\subset \cM$ be the subset of ergodic measures. Recall that $\cM$ is a compact convex metrizable topological space. Given $\mu\in \cM$ we denote by $h_\mu(f)$ the measure-theoretic entropy of $\mu$, see  \cite{Wal:81} for the definition and details. Recall that if $f$ is expansive, then the entropy map $\mu\mapsto h_\mu(f)$ is upper semi-continuous. In particular, if $f$ is a  shift map then $\mu\mapsto h_\mu(f)$ is upper semi-continuous.
We denote by $\Per_n(f)$ the set of periodic points of $f$ with prime period $n$  and by $\Per(f)$ and the set of periodic points of $f$. For $x\in \Per_n(f)$,  the unique invariant measure supported on the orbit of $x$ is given by 
$
\mu_x=\frac{1}{n}(\delta_x+\dots +\delta_{f^{n-1}(x)}), $ where
 $\delta_y$ denotes the  Dirac measure on $y$. We also call $\mu_x$ the periodic point measure of $x$. Obviously, $\mu_x=\mu_{f^l(x)}$ for all $l\in\bN$. We write $\cM_{\rm Per}=\{\mu_x: x\in \Per(f)\}$ and observe that $\cM_{\rm Per}\subset \cM_E$.  Let $\Phi=(\phi_1,\dots,\phi_m):X\to \bR^m$ be a continuous $m$-dimensional  potential. For $\mu\in \cM$ we define $\rv(\mu)=(\int \phi_1\, d\mu,\dots, \int \phi_m\, d\mu)\in \bR^m$ and 
 \[
 \cR(\Phi)=\{\rv(\mu): \mu\in \cM\}.
 \]
 It follows from the definitions that $\cR(\Phi)$ is a compact convex subset of $\bR^m$. The set $\cR(\Phi)$ is often called the  rotation set of the potential $\Phi$, see \cite{GM,Je,KW} for details. The entropy spectrum of $\Phi$ on $\cR(\Phi)$ is defined by
 \[
 \cH(w)=\cH_{f,\Phi}(w)=\sup\{h_\mu(f): \rv(\mu)=w\}.
 \]
  We frequently omit the dependence of $\cH_{f,\Phi}$ on $f$ and/or $\Phi$ when it is clear from the context which map $f$ and/or potential $\Phi$ is meant. 
\subsection{Discontinuous entropy spectra for shift maps}\label{sec:disc_entropy}
Next we review results from \cite{Wo} where  the discontinuity of entropy spectra for shift maps is proven.\footnote{Here we present a slightly modified version of the example in \cite{Wo}  with the goal to ease the presentation of this paper.}   
First we introduce certain sets in $\bR^2$  to define the family of $2$-dimensional potentials with a discontinuous entropy spectrum.
Fix $a,b\in \bR$ with $\log 3 <a<b$ and let $\alpha\in\bN$. Fix $\theta\in (0,1)$. We consider a strictly increasing and strictly concave (and hence continuous) function $h:[a,b]\to\bR$ with $h(a)>\log 3$. We define $w_\infty=(a,h(a))$ and let $(x_\ell)_{\ell \geq 1}$ be a strictly decreasing sequence with $x_\ell \in (a,b)$  such that $v_\ell\eqdef(x_\ell,h(x_\ell))$ satisfies
$ ||v_\ell-w_\infty||<C\theta^\ell$
for all $\ell\geq 1$ and some $C>0$.   We define $u_\ell=(x_\ell,h(a))$ for all $\ell\geq 1$. 
  Let $w_0=(b,h(a))$. Further, for $\ell\geq 1$ we define 
  \begin{equation}\label{defwk}
  w_\ell=\frac{1}{\ell+\alpha+1}\left((\alpha+1) w_0+\sum\limits_{j=1}^{\ell}v_j\right).
  \end{equation}
Define
$\cV=\{w_\ell: k\geq 0\}\cup     \{w_\infty\}$.
Further, let $\cR=\conv(\cV)$ denote the convex hull of $\cV$. 
 It follows that $\cR$ is a compact set whose boundary  is an infinite polygon with extreme point set $\cV$. 
We refer to Figure 1 for an illustration.
\begin{figure}
\includegraphics[width=3in]{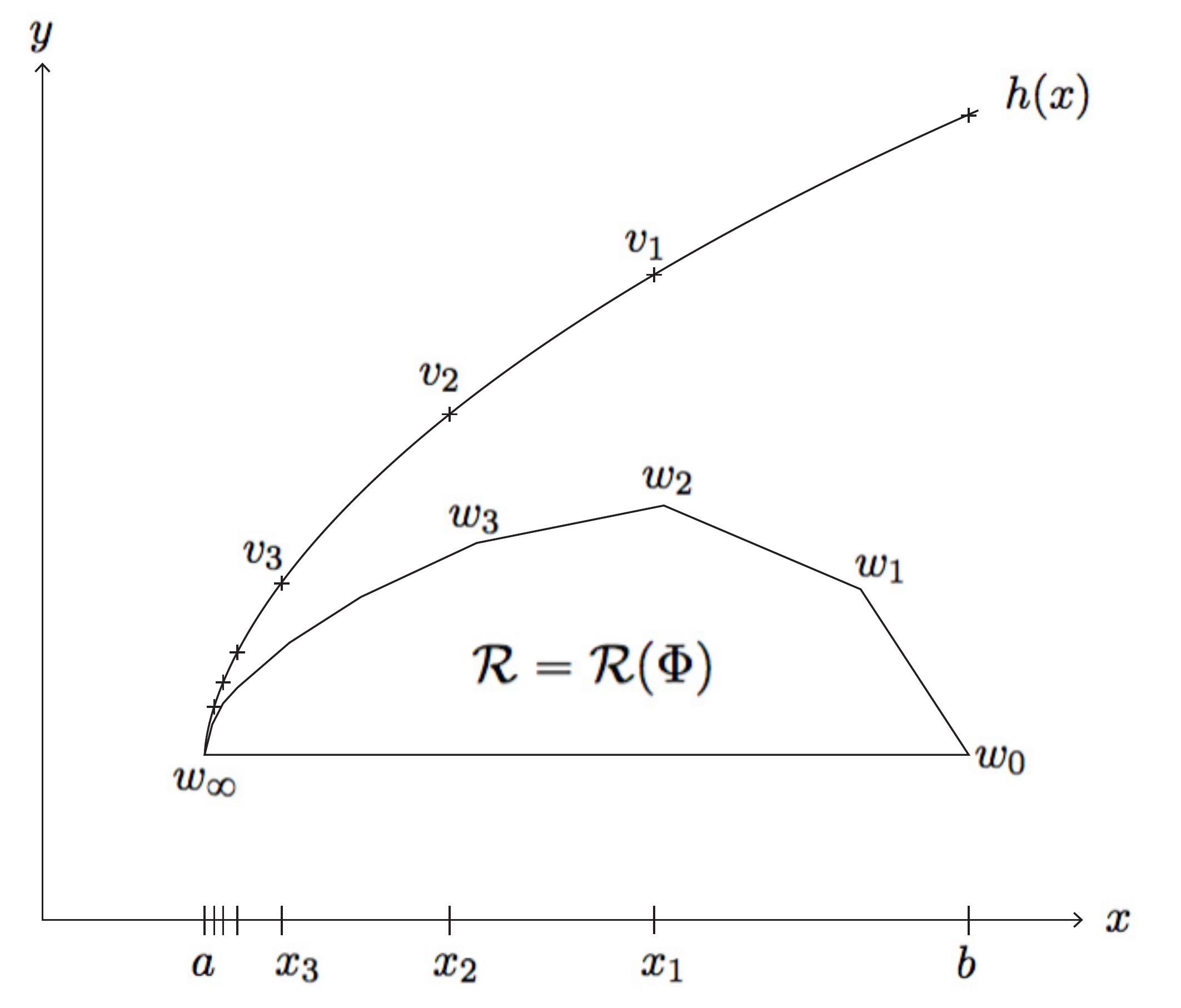}
\caption{The set $\cR=\cR(\Phi)$ generated by the function  $h:[a,b]\to\bR$ and the sequence $(x_\ell)_{\ell\geq 1}.$  }
\end{figure}

Let $\sigma:\Sigma^+_3\to \Sigma^+_3$ be the one-sided full shift with alphabet $\{0,1,2\}$  endowed with the $\theta$-metric. We shall construct a potential  $\Phi:\Sigma^+_3\to \bR^2$ which depends on the function $h$ and the sequence of points $(x_k)_{k\geq 1}$ as follows:   First, we define several subsets of $\Sigma^+_3$.  Let $S=\{0,1\}$. Let $k\in \bN$. We define
\begin{equation}\label{defphisets}
\begin{aligned}
X(k)&=\{\xi \in \Sigma^+_3: \xi_0,\dots,\xi_{k-1}\in S, \xi_{k}=2\},\\
X_0(\alpha)&=\bigcup_{k=0}^{\alpha} X(k),\\
X(\infty)&=\{\xi \in \Sigma^+_3: \xi_k\in S \ \mbox{for all} \ k\in \bN\}=S^{\bN}.
 \end{aligned}
\end{equation}
Note that $X(0)=\cC_0(2)=\{\xi \in \Sigma^+_3: \xi_0=2\}$.
We define $\Phi:\Sigma_3^+\rightarrow\mathbb{R}^2$ by
\begin{equation}\label{defphipot}
\Phi(\xi)=\begin{cases}
w_0& {\rm if}\,\,   \xi\in X_0(\alpha)\\
                      u_{k-\alpha}   & {\rm if}\,\,  \xi\in X(k)\setminus \cC_{k-1}(1^{k})\,  , \, k>\alpha\\
                       v_{k-\alpha} & {\rm if} \, \, \xi\in X(k)\cap \cC_{k-1}(1^{k}),\, k>\alpha\\
                       w_\infty & {\rm if}\,\,  \xi\in X(\infty)\, 
            \end{cases}
\end{equation}

For $\ell\geq 1$ we consider the periodic point $\xi^\ell=\Or(1^{\ell+\alpha}2)$, that is $\xi^\ell$ is the periodic point whose generating segment $\tau_\ell\eqdef \tau_{\xi^\ell}$ is given by $\ell+\alpha$ 1's followed by a $2$. Hence $\xi^\ell\in\Per_{\ell+\alpha+1}(\sigma)$.
It follows from the definition of $\Phi$ (see \eqref{defphipot}) and an elementary computation  that $\rv(\mu_{\xi^\ell})=w_\ell$.
We have the following.
\begin{theorem}[\cite{Wo}]\label{thmwo}
Let $\Phi:\Sigma^+_3\to \bR^2$ be defined as in \eqref{defphipot}. Then
\begin{enumerate}
\item[\rm (i)] 
 The potential $\Phi$ is Lipschitz continuous and $\cR(\Phi)=\conv(\cV)$. Moreover, for each $\ell\in \bN\cup\{\infty\}$ the point $w_\ell$ is an extreme point of  $\cR(\Phi)$ and  $\lim_{\ell\to \infty} w_\ell=w_\infty$;
 \item[\rm (ii)]    For each $\ell\in\bN$ with $\ell\geq 1$, $\cM_\Phi(w_\ell)\eqdef \{\mu\in \cM:\rv(\mu)=w_\ell\}=\{\mu_{\xi^\ell}\}$, in particular $\cH(w_\ell)=0$;
\item[\rm (iii)]  $\Phi^{-1}(\{w_\infty\})=\{0,1\}^{\bN}$ and $\cH(w_\infty)=\log 2$. In particular, $\cH$ is not continuous at $w_\infty$.
 \end{enumerate}
\end{theorem}

\
\section{Discontinuity of the entropy spectrum of the Lyapunov exponents for horseshoes}\label{sec:2}

This section is devoted to the construction of a horseshoe map $f$ with basic set $\Lambda$ such that the potential $\Phi_L=\Phi_{L,f}:\Lambda\to \bR^2$ given by
\begin{equation}
\Phi_{L}( x)= (-\log \|Df(x)\vert_{E^s_x}\|,\log \|Df(x)\vert_{E^u_x}\|)
\end{equation}
 induces  the  
potential $\Phi$ defined in Equation \eqref{defphipot} (see Theorem \ref{thm32} below for the precise statement). By $\Lambda$ being a basic set of $f$ we mean that $\Lambda$ is a locally maximal hyperbolic set of $f$ such that $f\vert_\Lambda$ is topologically mixing.
We  shall construct $f$ as the limit of a sequence $(f_k)_{k \in \bN}$ of piecewise-linear horseshoe maps such that each $f_k\vert_{\Lambda_k}$ is conjugate to the full shift on $\Sigma_3^\pm$.  More precisely, we prove  following result.
\begin{theorem}\label{thm32}There exists a $C^2$-horseshoe map $f$ with basic  set $\Lambda$ such that $f\vert_\Lambda$ is topologically conjugate  to $\sigma:\Sigma_3^\pm\to \Sigma_3^\pm$ via  $g: \Lambda\to \Sigma_3^\pm$, and a map $\Phi:\Sigma^+_3 \to \bR^2$ satisfying the conditions in equation \eqref{defphipot} such that 
\begin{equation}\label{eq888}
\Phi_L(x)=\Phi\circ \Pi\circ g(x)  \quad \mbox{ for all }\,\, x\in \Lambda.
\end{equation}

\end{theorem}
 

First we consider a piecewise linear horseshoe map $f_0$ on three symbols with a constant expansion rate $\lambda_\infty > 3$ and  two distinct contraction rates $\delta_\infty, \delta_0 \delta_\infty<\frac{1}{3} $. We denote by ${\Lambda}_0$ the basic set of ${f}_0$. Each $f_k$ is defined as a small perturbation of the previous horseshoe map $f_{k-1}$. For each $k \geq 1,$ we define finitely many disjoint subsets of the unit square where we change the expansion and contraction rates of the map, thereby adding finitely many values in the range of the derivative $Df_k$ on $\Lambda_k$. We design the corresponding perturbations, which we call  surgeries, to be $C^2$-diffeomorphisms which preserve piecewise linearity on $\Lambda_{k-1}$. We then show that  $f=\lim_{k\to\infty} f_k$ is a $C^2$-horseshoe map that satisfies the assertions in our Main Theorem in Section 1.2. \\

\subsection{Construction of the surgeries}
~Let $\cS=[0,1]\times[0,1]\subset \bR^2$, and let $\l_\infty>3$ and $0<\d_\infty<\frac{1}{3}$ be fixed. We consider a $C^2$-diffeomorphism $\widetilde{f}_0$ defined in a neighborhood of $\cS$ onto its image in $\bR^2$ such that $\widetilde{f}_0$ is a piecewise linear $3$-fold horseshoe map on $\cS$ with constant
vertical expansion rate $\lambda_\infty$ and constant horizontal contraction rate $\delta_\infty$.\footnote{In the definition of $\lambda_\infty$ and $\delta_\infty$ we use the infinity subscript  since these values will be related to the point $w_\infty$ in the symbolic example in Theorem \ref{thmwo}.} We refer to Figure 2 for an illustration.  
\begin{figure}
\includegraphics[width=4in]{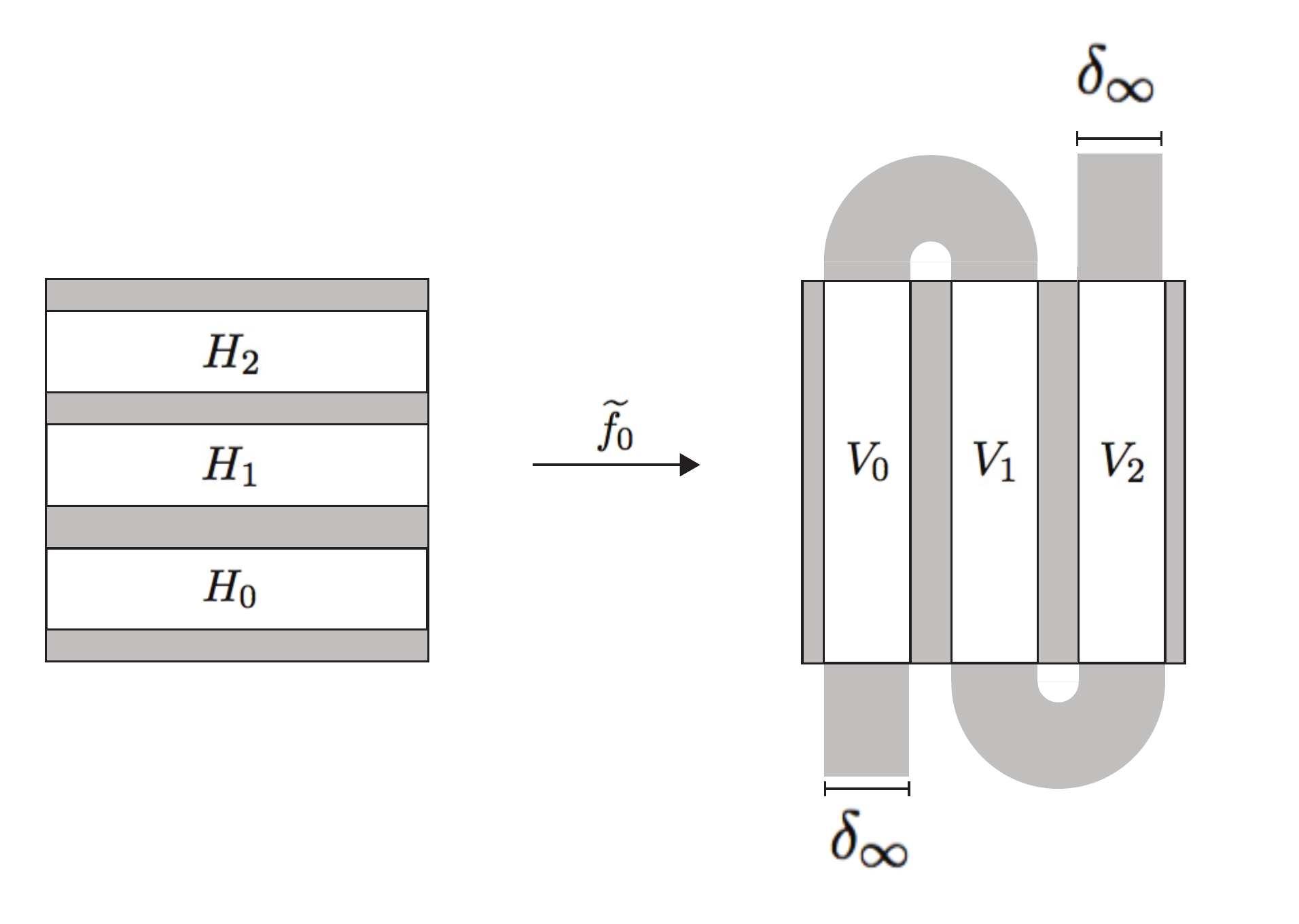}
\caption{Schematic representation of the standard horseshoe map on three symbols, $\widetilde{f_0}.$ Note that $k$ forward and backward iterations of $\widetilde{f_0}$ constrains $\widetilde{\Lambda_0}$ to $3^{2k}$ disjoint boxes in the unit square.}
\label{linearhorseshoe}
\end{figure}
Clearly,  the horizontal strips of size $1\times \l_\infty^{-1}$, denoted by $H_i$, are mapped  linearly onto the vertical strips, denoted by $V_i$, of size $\d_\infty\times 1$ for $i=0,1,2$.  It follows that  $\widetilde{\Lambda}_0=\{x\in \cS: \widetilde{f}_0^n(x)\in \cS  \mbox{ for all  }\, n\in \bZ\}$ is a basic set of $\widetilde{f_0}$. Moreover, the map $\widetilde{g_0}: \widetilde{\Lambda_0}\to \Sigma^\pm_3$, defined by $\widetilde{g_0}(x)_n=\xi_i$ iff $\widetilde{f_0}^n(x)\in H_i$,
provides a conjugacy between $\widetilde{f_0}\vert_ {\widetilde{\Lambda_0}}$ and $\sigma:\Sigma^\pm_3\to \Sigma^\pm_3$.



Given $C^2$-functions $F$ and $\tilde{F}$ defined in  a neighborhood of $\cS$ we denote by $\Vert F-\tilde{F}\Vert_2=\Vert F-\tilde{F}\Vert_{2,\cS}$ the $C^2$-distance on $\cS$ of $F$ and $\tilde{F}$. 

Fix $\alpha\in\bN$  (cf. Section 3.2) and $0<\delta_0<1$. Recall the definition of the set $X_0(\alpha)\subset \Sigma_3^+$ in equation \eqref{defphisets}. In the following we use the notation $\text{diag}(a,b)$ for the $2\times 2$ diagonal matrix with diagonal entries $a$ and $b$.
To obtain $f_0$ we perform a $C^2$-perturbation  on $\widetilde{f}_0$ such that $f_0$ is a piecewise linear horseshoe map with the following properties:

\begin{enumerate}
\item[(1)] $f_0$ has a basic set $\Lambda_0\subset \inn \cS$ and $f_0\vert_{\Lambda_0}$ is topologically conjugate via $g_0:\Lambda_0\to \Sigma^\pm_3$ to $\sigma:\Sigma^\pm_3\to \Sigma^\pm_3$;
\item[(2)] $Df_0(x)=\pm\text{diag}(\d_0\d_\infty, \l_\infty)$, with the same sign as $D\widetilde{f}_0(y)$ whenever $g_0(x)=\widetilde{g_0}(y)$ and  $(\Pi\circ g_0)(x)\in X_0(\alpha)$;
\item[(3)]  $Df_0(x)=D\widetilde{f_0}(y)$ whenever $g_0(x)=\widetilde{g_0}(y)$ and  $(\Pi\circ g_0)(x)\notin X_0(\alpha)$. 
\end{enumerate}

To construct $f_0$ we pick $\d_0$ sufficiently close to $1$ and change the contraction rate from $\delta_\infty$ to $\delta_0\delta_\infty$ on those horizontal strips in $\cS$ whose symbolic forward trajectories mirror those of points in $X_0(\alpha)$. We denote the union of the corresponding horizontal strips by $\cS_0(\alpha)$. Hence,
\begin{equation}\label{eqrett1}
\{x\in \Lambda_0: (\Pi\circ g_0)(x)\in X_0(\alpha)\}\subset \cS_0(\alpha).
\end{equation}

Recall that $f_0\vert_{\Lambda_0}$ is  $\cC^2$ structurally stable, i.e., there exists an $\epsilon_0>0$ such that if $F:\cS\rightarrow \RR^2$ is a $\cC^2$ map that satisfies $\Vert F-f_0\Vert_2<\epsilon_0$, then $F\vert_{\Lambda_F}$ is topologically conjugate to $f_0\vert_{\Lambda_0}$, where $\Lambda_F=\{x\in\cS: F^n(x)\in\cS \mbox{ for all  }\, n\in\ZZ\}$. Moreover, by  making $\epsilon_0$ smaller if necessary, we can assure that whenever $\Vert F-f_0\Vert_2<\epsilon_0$ with $F\vert_{\cS_0(\alpha)}= f_0\vert_{\cS_0(\alpha)}$, then
\begin{equation}\label{eqfin1}
 DF(y)=\pm\text{diag}(\d_0\d_\infty, \l_\infty)
 \end{equation}
 for each $y\in \Lambda_F$ that is conjugate to  a point $x\in \Lambda_0$ with $(\Pi\circ g_0)(x)\in X_0(\alpha)$ (cf. property (2) in the definition of $f_0$).

Let $C>0$, $0<\theta<1$, $a=-\log\d_\infty$ and $b=-\log\d_0\d_\infty$. 
We fix  a strictly increasing and strictly concave function $h:[a,b]\to \bR$ with $h(a)=\log\l_\infty$ (cf. Section 2.3).

We will inductively construct a sequence of piecewise-linear horseshoe maps $f_k$, convergent in the $C^2$-norm, so that $f=\lim_{k\to\infty} f_k$ satisfies the assumptions in our Main Theorem.  More precisely, in order to construct $f_{k}$ we  perform perturbations (which we call surgeries) on $f_{k-1}$ on specific subsets of $\cS$ that contain all points $x\in \Lambda_{k-1}$ such that $\Pi\circ g_{k-1}(x)$ belongs to one of the following subsets of $\Sigma^+_3$ for $k\geq1$:
\begin{equation}
\begin{aligned}
U(k)&=X(k+\alpha)\backslash \cC_{k+\alpha-1}(1^{k+\alpha}),\\
V(k)&=X(k+\alpha)\cap \cC_{k+\alpha-1}(1^{k+\alpha}),
\end{aligned}
\end{equation}
where $g_{k-1}:\Lambda_{k-1}\to \Sigma^\pm_3$ is the corresponding conjugating map. The surgeries themselves will depend on the function $h$.

Suppose $f_{k-1}$ is a piecewise-linear horseshoe map with basic set $\Lambda_{k-1}=\bigcap_{j\in\ZZ}f^j_{k-1}(\cS)$ that is topological conjugate  to $\sigma:\Sigma_3^\pm\to \Sigma_3^\pm$ with conjugating map $g_{k-1}$. Define
\begin{equation}
W_k=\cS\cap f_{k-1}^{-(k+\alpha)}(\cS)\cap f_{k-1}^{k+\alpha}(\cS).
\end{equation}

Thus, $W_k$ contains $\ell_k=3^{2(k+\alpha)}$ disjoint closed horizontal rectangles $B_1,\dots, B_{\ell_k}$ such that $f_{k-1}$ is linear in each of the rectangles $B_i$ and 
\begin{equation}\label{eq8}
\Lambda_{k-1}\subset \bigcup_{i=1}^{\ell_k} \inn B_i.
\end{equation}
We further assume that $Df_{k-1}$ is a diagonal matrix on each of the rectangles $B_i$.
We refer to the rectangles $B_1,\dots, B_{\ell_k}$ as \emph{boxes} of level $k$ and write $\cB_k=\{B_i: i=1,\dots,\ell_k\}$.
Let $\cU_k$ be the collection of boxes  in $\cB_k$ such that $B\in \cU_k$ whenever $B\cap (\Pi\circ g_{k-1})^{-1}(U(k))\not=\emptyset$. Let $\cV_k$ be analogously defined, but the boxes in  $\cV_k$ satisfy $B\cap(\Pi\circ g_{k-1})^{-1}(V(k))\not=\emptyset$.
We define 
\begin{equation}\label{defgammak}
\gamma_k=  \frac{1}{3}\min_{B}d(\partial B, \Lambda_{k-1}\cap B),
  \end{equation}
  where the minimum is taken  over all boxes $B\in \cU_k\cup \cV_k$. By \eqref{eq8}, $ \gamma_k>0$.
For each of the boxes $B\in \cU_k\cup \cV_k$ we select  a closed concentric horizontal sub-box $B'\subset B$ such that $\mbox{dist}(B',\partial B)\geq2\gamma_k$ and $\mbox{dist}(B\setminus  B',\Lambda_{k-1}\cap B)\geq\gamma_k$. 
Here the notion of being concentric refers to the property that $B$ and $B'$ have the same center point.
We will make use of the following well-known result which is a consequence of the shadowing lemma and the structural stability of basic sets.

\begin{theorem}\label{thmclassic}
Let $F$ be a $\cC^2$-diffeomorphism defined on a neighborhood of $\cS$ with  basic set $\Lambda_F=\{x\in \cS: F^n(x)\in \cS\, \mbox{ for all  }\, n\in \bZ\}\subset \inn \cS$  Then for all $\delta>0$ there exists  $\epsilon>0$ such that if $\Vert F-\tilde{F}\Vert_2<\epsilon$, then $\tilde{F}$ has a basic set $\tilde{\Lambda}\subset \inn S$. Further, $F\vert_\Lambda$ is topological conjugate to $\tilde{F}\vert_{\tilde{\Lambda}}$ with conjugating map $G:\Lambda\to \tilde{\Lambda}$ such that 
$\Vert x-G(x)\Vert<\delta$ for all $x\in \Lambda$. In particular, $d_H(\Lambda, \tilde{\Lambda})<\delta$, where $d_H(.,.)$ is the Hausdorff metric of the set.
\end{theorem}
\noindent
Next we pick $d_k>0$ such that the $d_k$-neighborhood of $\Lambda_{k-1}$ is contained in 
\begin{equation}\label{defBprim}
\bigcup_{B\in \cU_k\cup \cV_k} B' \,\,\,\,\,\,\,\, \cup \,\,\,\,\,\, \bigcup_{B\in \cB_{k}\setminus (\cU_k\cup \cV_k)} B.
\end{equation}
We denote by $\cB_{k}'$ the collection of rectangles in \eqref{defBprim}.
 By applying Theorem \ref{thmclassic} to $f_{k-1}$ with $\delta=\frac{1}{2}\min\{\gamma_k, d_k\}$ we conclude that there exists $\epsilon>0$ such that  for any $C^2$ -perturbation $\epsilon$-close to $f_{k-1}$, the corresponding basic set $\Lambda$ satisfies 
 
 \begin{equation}\label{pw-linearity}
 \Lambda\subset \bigcup_{B\in \cB_{k}'} \inn B. 
 \end{equation}
 Next we select  $0<\delta_k<1$ and $\lambda_k>1$ such that for $x_k=-\log\d_k\d_\infty>a$ we have $h(x_k)=\log\l_k\l_\infty$ and  
 \begin{equation}\label{condi}
 \begin{aligned}
 \Vert  (x_k , h(x_k)) - (a, h(a))\Vert= & \, \, \Vert (-\log \delta_k, \log\lambda_k)\Vert < C\theta^k.
 \end{aligned}
 \end{equation}~\\
 By continuity of $h$ and the fact that $h$ is a strictly increasing function, such $\d_k$ and $\l_k$ exist. Moreover, the convergence of $\d_k$ to $1$ guarantees the convergence of $\l_k$ to $1$. 
 
 We define a local surgery map $\cL_k$ (of level $k$) depending on $\d_k$ and $\lambda_k$.
 Let  $\cL_k:\RR^2\rightarrow\RR^2$  be $C^2$-diffeomorphism with the following properties:

\begin{enumerate}[(i)]
\item $\cL_k\vert_{B'}$ is  linear for all sub-rectangles $B'$ of boxes  $B\in \cU_k\cup \cV_k$. More precisely, $\cL_k\vert_{B'}=\text{diag}(\d_k,\l_k)$ for all sub-rectangles $B'$ contained in boxes  in $\cV_k$, and $\cL_k\vert_{B'}=\text{diag}(\d_k,1)$ for all sub-rectangles contained in boxes in  $B\in\cU_k$.
\item $\cL_k$ is the identity map on $\RR^2\setminus \bigcup_{B \in \cU_k\cup\cV_k} B$, that is, the local surgeries are smoothed out in $B\setminus B'$ to the identity.
\item $\cL_k^{\pm 1}(B')$ is compactly contained in $B$ and $d(\partial\cL_k^{\pm 1}(B'), \partial B)>\gamma_k$ for all $B\in  \cU_k\cup \cV_k.$
\end{enumerate}
We refer to Figure 3 for an illustration.
\begin{figure}[ht]\label{fig3}
\includegraphics[width=4in]{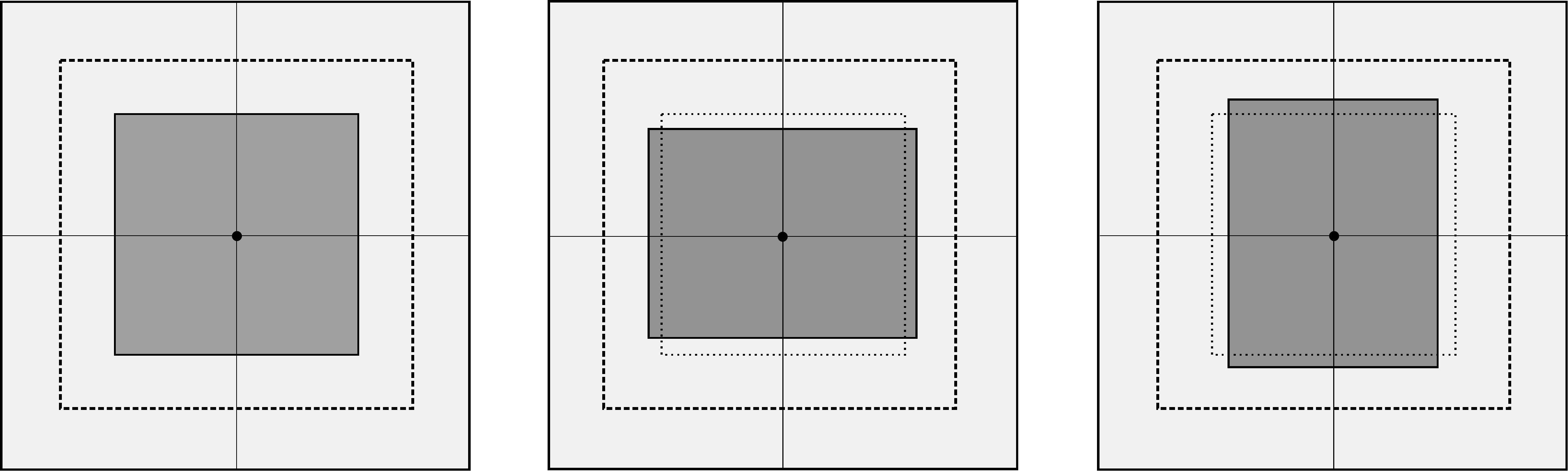}
\caption{The shaded box (left) is the concentric $B'$ contained in the larger, dotted $B$ box. Under the (local) linear perturbation  it remains in the  box $B$ under $\cL_k$ (middle) and under $\cL_k^{-1}$ (right).}
\end{figure}
We note that such a function $\cL_k$ can be obtained via piecewise polynomial interpolation, and conditions (i),(ii),(iii) and \eqref{condi} can be attained by choosing $\d_k,\l_k$ sufficiently close to $1$. 
 Let $I_2:\RR^2\to\RR^2$ denote the  identity map and   and let $Z_2\in L(\RR^2,L(\RR^2,\RR^2))$ denote the  zero map. Further, for a twice-differentiable function $F:\RR^2\to\RR^2$ we denote by $D^2F(x)$ the second derivative of  $F$ at $x$. It follows from the construction that it is possible to make $D\cL_k(x)$  arbitrarily close to $I_2$  and $D^2\cL_k(x)$ can be made arbitrarily close  $Z_2$. More precisely, we can assure the following property which follows from elementary arguments using (for example) polynomial interpolation.
\begin{property}\label{property1}
For all $\epsilon>0$ there exists $\delta>0$ such that for $1-\delta<\delta_k<1<\lambda_k<1+\delta$ there exists a $C^2$-diffeomorphism $\cL_k:\RR^2\rightarrow\RR^2$  satisfying  {\rm (i),(ii),(iii)} such that $\Vert \cL_k(x)-x\Vert, \Vert D\cL_k(x)-I_2\Vert, \Vert D^2\cL_k(z)- Z_2\Vert<\epsilon$ for all $x\in \bR^2$, and condition \eqref{condi} holds.
\end{property}
We leave the details of the construction of the map $\cL_k$ satisfying Property \ref{property1} to the interested reader. Finally, we define $f_k=f_{k-1}\circ \cL_k$.


\begin{proposition}\label{pro34}
There exist constants $0<\d_k<1<\l_k$ such that condition $\eqref{condi}$ is satisfied, and $f_k$ is a $3$-fold horseshoe map with basic set $\Lambda_k\subset \inn \cS$ such that $f_k\vert_{\Lambda_k}$ is piecewise linear in a neighborhood of $\Lambda_k$ with diagonal matrix derivatives. Moreover, $f_k\vert_{\Lambda_k}$ is topologically conjugate to $f_{k-1}\vert_{\Lambda_{k-1}}$ and $\Vert f_k-f_{k-1}\Vert_2<\left(\frac{1}{2}\right)^{k+1}\epsilon_0$
\end{proposition}
\begin{proof}
First we notice that $f_{k}$ and $f_{k-1}$  only differ on the boxes in $ \cU_k\cup\cV_k$. Therefore, it suffices to consider a fixed box $B\in \cU_k\cup\cV_k$. 
Further, the $C^2$-norm changes occurring in the boxes in $\cV_k$ are greater or equal than the corresponding changes in the boxes in $\cU_k$  because of the additional expansion in the vertical direction. 
Thus, without loss of generality we may assume $B\in \cV_k$.
Let $x\in B$, and let $q=\cL_k(x)$. Applying the Mean Value Theorem yields
\begin{equation} \label{eq111}
\begin{split}
\Vert f_k(x)-f_{k-1}(x)\Vert &=  \Vert f_{k-1}(\cL_k(x))-f_{k-1}(x)\Vert \\
&\leq M_0 \Vert\cL_k(x)-x\Vert,
\end{split}
\end{equation}
 where $M_0$ is an upper bound of $\Vert Df_{k-1}\Vert$ on $\cS$. It follows from Property \ref{property1} that $||\cL_k(x)-x||$ can be made arbitrarily small provided $\d_k$ and $\l_k$ are sufficiently close to $1$. 
For the  derivatives  of $f_k$ and $f_{k-1}$, we have
\begin{equation}\label{eq122}
\begin{split}
\Vert Df_{k}(x)-Df_{k-1}(x)||&\leq
 \Vert Df_{k-1}(q)\circ D\cL_k(x)-Df_{k-1}(x)\circ D\cL_k(x)\Vert \\
&+ \Vert Df_{k-1}(x)\circ D\cL_k(x) - Df_{k-1}(x)\circ I_2\Vert \\
&\leq \Vert D\cL_k(x)\Vert \cdot \Vert Df_{k-1}(q)-Df_{k-1}(x)\Vert  \\
&+ \Vert Df_{k-1}(x)\Vert\cdot \Vert D\cL_{k-1}(x)-I_2\Vert\\
&\leq M_1  \Vert Df_{k-1}(q)-Df_{k-1}(x)\Vert + M_0  \Vert D\cL_{k}(x)-I_2\Vert,
\end{split}
\end{equation}
where $M_1$ is an upper bound for $\Vert D\cL_k\Vert$ on $\cS$. Again, by Property \ref{property1} for $\delta_k$ and $\lambda_k$ sufficiently close to $1$, it is possible to make $\Vert q-x\Vert=\Vert \cL_k(x)-x\Vert $ and  $\Vert D_x\cL_k-I_2||$
as small as necessary.
 Finally we consider the second derivatives of $f_k$ and $f_{k-1}$. Let $M_2$ be an upper bound for $D^2f_{k-1}$ on $\cS$. 
 We have 
 \begin{equation}\label{eq123}
 \begin{aligned}
 & \Vert D^2f_k(x)-D^2f_{k-1}(x)\Vert   =    \Vert D^2(f_{k-1}\circ \cL_k)(x) -D^2f_{k-1}(x)\Vert \\
  =& \Vert D^2f_{k-1}(q)( D\cL_k(x)\circ D\cL_k(x))+ Df_{k-1}(q)\circ D^2\cL_k(x)-D^2f_{k-1}(x)\Vert \\
  \leq &  2 M_2\Vert D\cL_k(x)\circ D\cL_k(x)-I_2\circ I_2\Vert+M_1 \Vert D^2\cL_k(x)\Vert,
 \end{aligned}
 \end{equation}    
 provided $\Vert q-x\Vert$ is sufficiently small which can be assured by Property \ref{property1}.
  Moreover, again by Property \ref{property1} (i.e. by selecting $\delta_k$ and $\lambda_k$ sufficiently close to $1$) it is possible to make $\Vert D\cL_k(x)-I_2\Vert$ arbitrarily small. By the same argument $\Vert D^2\cL_k(x)-Z_2\Vert$ can be made arbitrarily small. Let $\epsilon_k=\min\{\epsilon,\left(\frac{1}{2}\right)^{k+1}\epsilon_0\}$, where $\epsilon$ is chosen to satisfy condition \eqref{pw-linearity}. It now follows from equations \eqref{eq111},\eqref{eq122},\eqref{eq123} that there exists $\delta_k<1<\lambda_k$  and a corresponding map $\cL_k$ so that $\Vert f_k-f_{k-1}\Vert_2<\epsilon_k$ holds. We use this surgery map $\cL_k$ to define $f_k$. We conclude that $f_k$  has a basic set $\Lambda_k\subset \inn\cS$ such that $f_k\vert_{\Lambda_k}$ is topologically conjugate to $f_{k-1}\vert_{\Lambda_{k-1}}$. Moreover, piecewise linearity of $f_k$ with diagonal matrix derivatives follows from the construction of $f_k$ and the fact that $\Lambda_k\subset \bigcup_{B\in\cB_{k}'} B$.
  \end{proof}
From now on we additionally assume (without loss of generality) that the following holds:
\begin{equation}
\d_{k_1}<\d_{k_2}<1<\l_{k_2}<\l_{k_1} \qquad \mbox{whenever} \quad k_1<k_2,
\end{equation}
\begin{equation}\label{lem:monotonedecreasing}
\epsilon_{k_2}<\epsilon_{k_2}\quad \mbox{and} \quad \gamma_{k_2}<\gamma_{k_1} \qquad \mbox{whenever} \quad k_1<k_2.
\end{equation}
Next we list several facts that are immediate consequences of the construction of the horseshoe maps $f_k$.
\begin{corollary}\label{cor23}
Let $1\leq k_1<k_2$ and let $B^i\in \cU_i\cup \cV_i$ for $i=k_1,k_2$. Then $B^{k_1}\cap B^{k_2}=\emptyset.$
Moreover, $(f_k)_k$ is a Cauchy sequence with respect to the  $C^2$-norm on $\cS$.
\end{corollary}

\begin{corollary}
The sequence $(f_k)_k$ converges in the $\cC^2$-norm on $\cS$ to a $3$-fold horseshoe  map $f$ with a basic set $\Lambda\subset \inn \cS$. Moreover, $f\vert_{\Lambda}$ is topologically conjugate to $f_0\vert_{\Lambda_0}$.
\end{corollary}
\begin{proof}
The statement that $f=\lim_{k\to\infty} f_k$ is a well-defined limit in the $C^2$-norm on $\cS$ follows from Corollary \ref{cor23}. To establish the topological conjugacy we note that 
\begin{equation}
\Vert f-f_0\Vert_2\leq \sum_{k=1}^\infty \Vert f_k-f_{k-1}\Vert_2\leq \sum_{k=1}^\infty \epsilon_k\leq \frac{1}{2}\epsilon_0, 
\end{equation}
which proves the claim.
\end{proof}
Let $g:\Lambda\to  \Sigma_3^\pm$ denote the conjugating map between $f\vert_\Lambda$ and $\sigma: \Sigma_3^\pm\to \Sigma_3^\pm$.


\begin{corollary}\label{cor:derivativeoflimit}
Let $1\leq k_1<k_2$. Suppose $x\in \Lambda, y_1\in \Lambda_{k_1}$ and $y_2\in \Lambda_{k_2}$ with $g(x)=g_{k_1}(y_1)=g_{k_2}(y_2)$. Then $\Vert x-y_1\Vert< \frac{1}{2}\gamma_{k_1}$ and $\Vert y_1-y_2\Vert< \frac{1}{2}\gamma_{k_1}$.

\end{corollary}
For $x\in \Lambda$ we write $\overline{x}^+=\Pi\circ g(x)\in \Sigma^+_3$. Similarly, for $x\in \Lambda_k$ we write  $\overline{x}^+=\Pi\circ g_k(x)\in \Sigma^+_3$.

\begin{lemma}
Let $x\in \Lambda$ and $y\in \Lambda_k$ with $\overline{x}^+=\overline{y}^+$and  $\overline{x}^+\in U(k)\cup V(k)\cup X_0(\alpha)$. Then $Df(x)=Df_k(y)$. Moreover,
\begin{enumerate}
	\item[{\rm (1)}] if $\overline{x}^+\in U(k)$, then $Df(x)=\pm\text{diag}(\d_k\d_\infty,\l_\infty)$,
	\item[{\rm (2)}] if $\overline{x}^+\in V(k)$, then $Df(x)=\pm\text{diag}(\d_k\d_\infty,\l_k\l_\infty)$, and
	\item[{\rm (3)}] if $\overline{x}^+\in X_0(\alpha)$, then $Df(x)=\pm\text{diag}(\d_0\d_\infty,\l_\infty)$.
\end{enumerate}
\label{lem:limithorseshoederivative}
\begin{proof}
We first prove (1) and (2).
It follows from the construction (i.e. from equation \eqref{lem:monotonedecreasing}) that $x$ and $y$ belong to sub-rectangles $B'_x$ and $B'_y$ of $\cB'_{k}$ on which $Df_k$ coincides. This implies $Df(x)=Df_k(y)$. Therefore, statements (1) and (2) follow from $f_k=f_{k-1}\circ \cL_k$, condition (i) in the definition of $\cL_k$ and Corollary \ref{cor23}. 
If $\overline{x}^+\in X_0(\alpha)$, then  $Df(x)=Df_0(y)=\pm\text{diag}(\d_0\d_\infty,\l_\infty)$ follows from equation \eqref{eqfin1}.
\end{proof}
\end{lemma}

\subsection{The Horseshoe in the limit}
Let now $f=\lim f_k$ with basic set $\Lambda$ and and conjugating map $g:\Lambda\to \Sigma_d^\pm$. For $x\in \Lambda$ we recall the notation  $\overline{x}^+=(\Pi\circ g)(x)$. 

\begin{theorem}\label{thmwich}
	$\Phi_L(x)=\Phi_L(y)$ for all $x,y\in \Lambda$ with $\overline{x}^+=\overline{y}^+$.
\end{theorem}
\begin{proof}
	Let $x,y\in \Lambda$ with $\overline{x}^+=\overline{y}^+$. Then one of the following cases occur: either $\overline{x}^+\in X_0(\alpha)$, $\overline{x}^+\in U(k)\cup V(k)$ for some $k\in \bN$, or $\overline{x}^+\in X(\infty)$. In the first and second cases we conclude from Lemma \ref{lem:limithorseshoederivative} that $\Phi_L(x)=\Phi_L(y)$. Next we consider the case $\overline{x}^+\in X(\infty)$. Since $g$ is a topological conjugacy there exists  points $x_k, y_k\in \Lambda$ with $\overline{x_k}^+=\overline{y_k}^+\in U(k)\cup V(k)$ such that $\lim x_k=x$ and $\lim y_k=y$. Hence, by the previous case, $\Phi_L(x_{k})=\Phi_L(y_{k})$. Finally,  $\Phi_L(x)=\Phi_L(y)=w_\infty$ follows follows from the continuity of $\Phi_L$ on $\Lambda$.
\end{proof}

Let $k\in \bN\setminus\{0\}$. We define the following  points in $\RR^2:$
\begin{align*}
&v_{k}=(x_{k},h(x_{k}))=(-\log\d_k\d_{\infty},\log\l_\infty\l_{k}), \\ &u_{k}=(x_{k},h(a))=(-\log\d_\infty\d_{k},\log \l_\infty),\\ 
&w_0=(b,h(a))=(-\log\d_0\d_\infty,\log\l_\infty), \\
&w_\infty=(a,h(a))=(-\log\d_\infty,\log\l_\infty).
\end{align*}
We now define the potential $\Phi:\Sigma_3^+\to\bR^2$ in \eqref{defphipot} by using this collection of points.

 By construction, condition $\eqref{condi}$  holds for all $k$. Thus, Theorem \ref{thmwo} applies to the potential $\Phi$. We are finally in a position to prove Theorem \ref{thm32}.

\begin{proof}[Proof of Theorem \ref{thm32}]
	The only thing that remains to prove is that $\Phi_L(x)=\Phi(\overline{x}^+)$ for all $x\in \Lambda$.
	Let $x\in \Lambda$. If $\overline{x}^+\in X_0(\alpha)$, then $Df(x)=\pm\text{diag}(\d_0\d_\infty,\l_\infty)$, and thus $\Phi_L(x)=w_0=\Phi(\overline{x}^+)$.
	Next, we assume $\overline{x}^+\in X(k+\alpha)$ for some $k>0$. If $\overline{x}^+\in U(k)$, then $\Phi_L(x)=u_{k}=\Phi(\overline{x}^+)$. Otherwise, $\overline{x}^+\in V(k)$. In which case, $\Phi_L(x)=v_{k}=\Phi(\overline{x}^+)$. Finally, if $\overline{x}^+\in X(\infty)$, we have that $\Phi_L(x)=w_\infty=\Phi(\overline{x}^+)$ which was already shown in the proof of Theorem \ref{thmwich}.
\end{proof}

\section{Proof of the Main Theorem}
Let $\cS\subset \bR^2$ be the unit square and let $f:\cS\to \bR^2$ denote the $\cC^2$ horseshoe map in Theorem \ref{thm32} with its basic set $\Lambda\subset \inn S$. We use the notation $f=f|_\Lambda$. Recall that $f$ is topologically conjugate to the shift map $\sigma:\Sigma_3^\pm\to \Sigma_3^\pm$ via $g:\Lambda\to \Sigma^\pm_3$. Further recall the definition of $\Phi_L=(-\phi^s,\phi^u): \Lambda \to\bR^2$ where $\phi^{s}=\log \|Df|_{E^{s}}\|$ and $\phi^{u}=\log \|Df|_{E^{u}}\|$. For $x\in \Lambda$ we write $\overline{x}=g(x)\in \Sigma_3^\pm$ and $\overline{x}^+=\Pi\circ g(x)\in \Sigma^+_3$. Let $\Phi: \Sigma_3\to \bR^2$ be the potential defined in Theorem \ref{thm32} (with respect to $f$).
Recall that $\Lambda$ is associated with a certain  function $h:[a,b]\to\bR$ and a strictly decreasing sequence $X=(x_\ell)_\ell\subset (a,b)$. By construction,  $h$ and   $X$ determine $\Phi_L$. 
 It is straightforward to verify that  $h$ and  $X$  also determine
 $\Phi$ by definition \eqref{defphipot}. We conclude that Theorem \ref{thmwo} holds for $(\sigma^+,\Phi)$, and thus $w\mapsto \cH_{\sigma^+}(w)=\cH_{\sigma^+,\Phi}(w)$ is discontinuous at the non-isolated extreme point $w_\infty$ of $\cR(\Phi)$. Therefore, our Main Theorem (see Section 1.2) is a consequence of the following result.
 \begin{proposition}
 $\cR_L(f)=\cR(\Phi)$ and $\cH_{f}(w)=\cH_{\sigma^+}(w)$ for all $w\in \cR_L(f)$.
 \end{proposition}
\begin{proof}
First, we consider  the potential $\Phi=(\phi_1,\phi_2)$. For $p,q\in \bR$ let $\mu^+_{p,q}$ denote the (unique) equilibrium measure of the
 potential $p\phi_1+q\phi_2$. We write $w_{p,q}=\rv(\mu^+_{p,q})$. It follows from \cite[Corollary 2]{Je} that 
\begin{equation}\label{eq11}
\inn \cR(\Phi)\subset \{w_{p,q}: p,q\in \bR\}.
\end{equation}
Furthermore, for  $p,q\in \bR$ and $\mu^+\in \cM_{\sigma^+}$ with $\rv(\mu^+)=w_{p,q}$,  the variational principle (also using the uniqueness of the equilibrium measure) implies that
\begin{equation}\label{eq22}
\cH_{\sigma^+}(w_{p,q})=h_{\mu^+}(\sigma^+)\quad\mbox{if and only if}\,\, \mu^+=\mu^+_{p,q}.
\end{equation}
We define $T:\cM_{\sigma^+}\to \cM_{f}$ by 
\begin{equation}\label{eq33}
T(\mu^+)(A)=\mu^+(\Pi\circ g(A)).
\end{equation}
Next, we consider the potential $-p\phi^s+q\phi^u:\Lambda\to \bR$ and the corresponding equilibrium measure $\mu_{p,q}\in \cM_{f}$.  It is well-known (see e.g.  \cite{Bo1}) that Theorem \ref{thmwich} and equation \eqref{eq888}  imply that $\mu_{p,q}=T(\mu_{p,q}^+)$ for all
$p,q\in\bR$. Again by  \cite[Corollary 2]{Je},  equations \eqref{eq11} and \eqref{eq22}  hold for $f$ and $\Phi_L$. We conclude that 
\begin{equation}\label{eq44}
\inn \cR_{L}(f)=\inn \cR_{\sigma^+}(\Phi)\, \mbox{and}\,\, \cH_f\vert_{\inn \cR_{L}(f)}=\cH_{\sigma^+}\vert_{\inn \cR(\Phi)}.
\end{equation}
The claim $\cR_{L}(f)=\cR(\Phi)$ follows from the first statement in  \eqref{eq44} and the fact that every convex set with non-empty interior is the closure of its interior. 
Note that both $\cH_{f}$ and $ \cH_{\sigma^+}$ are concave and upper semi-continuous functions and thus are continuous on line segments. Therefore, we may conclude from the second statement in 
\eqref{eq44} that $\cH_{f}(w)=\cH_{\sigma^+}(w)$ also holds for all $w\in \partial \cR_{L}(f)=\partial \cR(\Phi^+)$.
\end{proof}


\begin{thebibliography}{99}
\bibitem{BPS} L. Barreira, Y. Pesin and J. Schmeling, \emph{On a general concept of multifractality: multifractal spectra for dimensions, entropies, and Lyapunov exponents. Multifractal rigidity,} Chaos {\bf 7} (1997), 27--38.

\bibitem{BSS}L. Barreira, B. Saussol and J. Schmeling \emph{Higher-dimensional multifractal analysis}, Journal de Math\'ematiques Pures et Appliqu\'ees {\bf 9} (2002), 67--91.






499--513.

\bibitem{Bo}R. Bowen, \emph{Topological entropy for noncompact sets,} Transactions of the  American Mathematical Society  \textbf{184} (1973), 125--
136.

\bibitem{Bo1}R. Bowen, \emph{Equilibrium states and the ergodic theory of Anosov diffeomorphisms}, Second revised edition. With a preface by David Ruelle. Edited by Jean-Ren/'e Chazottes. Lecture Notes in Mathematics, 470. Springer-Verlag, Berlin, 2008. viii+75 pp.

\bibitem{C}V. Climenhaga, \emph{
Topological pressure of simultaneous level sets},
Nonlinearity {\bf 26} (2013), 241--268. 







\bibitem{GKR}D. Gale, V. Klee and R.T. Rockafellar,\emph{
Convex functions on convex polytopes,} 
Proceedings American Mathematical Society {\bf 19} (1968) 867--873. 


\bibitem{GM}W. Geller and M. Misiurewicz, \emph{Rotation and entropy}, Trans. Amer. Math. Soc., \textbf{351} (1999), 2927-2948.



\bibitem{FFW}A. Fan, D. Feng and J. Wu, \emph{Recurrence, dimension and entropy}, Journal of the London Mathematical Society {\bf 64} (2001),  229--244.


\bibitem{Je} O.~Jenkinson, \emph{Rotation, entropy, and equilibrium states}, Transactions of the  American Mathematical Society \textbf{353} (2001), 3713--3739.

\bibitem{KW}T. Kucherenko and C. Wolf, \emph{Geometry and entropy of generalized rotation sets}, Israel Journal of Mathematics \textbf{1999} (2014), 791-829. 





  
  \bibitem{Par}K.R. Parthasarathy,
{\it On the category of ergodic measures},
Illinois Journal of Mathematics \textbf{5} (1961), 648--656.




\bibitem{Wal:81} P.~Walters, \emph{An introduction to ergodic theory},
  Graduate Texts in Mathematics 79, Springer, 1981.

\bibitem{Wo}C. Wolf, \emph{A shift map with a discontinuous entropy function}, Discrete and Continuous Dynamical Systems, to appear, arXiv:1803.02440.


\bibitem{Z}K. Ziemian, \emph{Rotation sets for subshifts of finite type}, Fundamenta Mathematicae \textbf{146} (1995), 189--201.

\end{thebibliography}
\end{document}